\documentclass[10pt,twoside]{siamart1116}

\usepackage[english]{babel}
\usepackage{graphicx,epstopdf,epsfig}
\usepackage{amsfonts,epsfig,fancyhdr,graphics, hyperref,amsmath,amssymb}

\newcommand{\HE}{Name of Handling Editor}
\newcommand{\DoS}{Month/Day/Year}
\newcommand{\DoA}{Month/Day/Year}
\newcommand{\CA}{Name of Corresponding Author}

\newcommand{\Names}{G.J. Groenewald, D.B. Janse van Rensburg, A.C.M. Ran, F. Theron and M. van Straaten}
\newcommand{\Title}{Polar decompositions of quaternion matrices in indefinite inner product spaces}

\newtheorem{remark}[theorem]{Remark}
\newtheorem{example}[theorem]{Example}

\renewcommand{\theequation}{\thesection.\arabic{equation}}
\renewtheorem{theorem}{Theorem}[section]

\newcommand{\I}{\mathsf{i}}
\newcommand{\J}{\mathsf{j}}
\newcommand{\K}{\mathsf{k}}

\begin{document}

\bibliographystyle{plain}

\setcounter{page}{1}

\thispagestyle{empty}

 \title{\Title\thanks{Received
 by the editors on \DoS.
 Accepted for publication on \DoA. 
 Handling Editor: \HE. Corresponding Author: \CA}}

\author{
G.J.~Groenewald\thanks{Department of Mathematics and Applied Mathematics,
Research Focus: Pure and Applied Analytics, North-West~University,
Private~Bag~X6001,
Potchefstroom~2520,
South Africa.
(gilbert.groenewald@nwu.ac.za, dawie.jansevanrensburg@nwu.ac.za, frieda.theron@nwu.ac.za, madelein.vanstraaten@nwu.ac.za). Supported by a grant from DSI-NRF Centre of Excellence in Mathematical and Statistical Sciences (CoE-MaSS).}
\and
D.B.~Janse~van~Rensburg\footnotemark[2]
\and
A.C.M.~Ran\thanks{Department of Mathematics, Faculty of Science, VU University Amsterdam, De Boelelaan
    1111, 1081 HV Amsterdam, The Netherlands
    and Research Focus: Pure and Applied Analytics, North-West~University,
Potchefstroom,
South Africa. (a.c.m.ran@vu.nl).}
\and 
F.~Theron\footnotemark[2]
\and
M.~van~Straaten\footnotemark[2]}

\markboth{\Names}{\Title}

\maketitle

\begin{abstract}
Polar decompositions of quaternion matrices with respect to a given indefinite inner product are studied. Necessary and sufficient conditions for the existence of an $H$-polar decomposition are found. In the process an equivalent to Witt's theorem on extending $H$-isometries to $H$-unitary matrices is given for quaternion matrices.
\end{abstract}




\section{Introduction} \label{intro-sec}
Polar decompositions of real and complex matrices with respect to a given indefinite inner product have been extensively studied
. Necessary and sufficient conditions for the existence of an $H$-polar decomposition in the real and complex case are given in \cite{BR}, whereas in \cite{BMRRR} a description of the matrices admitting this decomposition can be found. See also the literature referenced in \cite{BMRRR}. Special cases of $H$-polar decompositions are studied for example in \cite{BMRRR2,BMRRR2err,BMRRR3}, and in \cite{MRR} stability of $H$-polar decompositions is studied. The study of polar decompositions of quaternion matrices (in the standard inner product space) goes back to 1955, see the paper by Wiegmann \cite{Wiegmann}, and also Proposition~3.2.5(d) in \cite{Rodman} and \cite{Zhang}.

Perusal of these studies suggests the need of results on the extension of isometries via Witt's Theorem in the quaternion case; see \cite{BMRRR3} for the complex case. We also need results on the existence of $H$-selfadjoint square roots and in \cite{onsnr3} the general case was studied for quaternion matrices, i.e., $H$-selfadjoint $m$th roots where $m$ is any positive integer, building on results for the complex case as found in \cite{onsnr1,MRR}.

Let $H$ be an invertible Hermitian matrix with quaternion entries which defines the indefinite inner product $[\,\cdot\,,\cdot\,]$. It will be clear from the context which invertible Hermitian matrix (or indefinite inner product) is meant, when we have more than one inner product under consideration. For a given square quaternion matrix $X$ we wish to find necessary and sufficient conditions for the existence of an \textit{$H$-polar decomposition}, i.e., such that $X=UA$ where $U$ is an $H$-unitary matrix (unitary with respect to $[\,\cdot\,,\cdot\,]$) and $A$ is an $H$-selfadjoint matrix (selfadjoint with respect to $[\,\cdot\,,\cdot\,]$).

We work mostly with matrices in an indefinite inner product space which have quaternion entries. Basic theory of quaternion linear algebra can be found in various books and papers, see for example the book by Rodman, \cite{Rodman}, and \cite{Zhang,ZhangWei}. 

In Section~\ref{secPrelim} we present some preliminary results, definitions and notation which are necessary to follow this paper. The conditions for the existence of $H$-selfadjoint square roots of $H$-selfadjoint quaternion matrices make up Section~\ref{secSqRoots}. In Section~\ref{secLem1stSteps} we pave the way for the rest of the paper with an important result, which is crucial in our proof for polar decompositions. Section~\ref{secWittsThm} focuses on the theory of Witt extensions. We give the conditions which prescribe the existence of a Witt extension in the quaternion case as well as the form of such a Witt extension. These will be essential for the main results. We obtain necessary and sufficient conditions for the existence of an $H$-polar decomposition for a given quaternion matrix in Section~\ref{secMainPolarDecomp}.

\section{Preliminaries}\label{secPrelim}
Denote the skew-field of real quaternions by $\mathbb{H}$ and the set of all vectors with $n$ quaternion entries by $\mathbb{H}^n$. This set $\mathbb{H}^n$ is considered as a right vector space, therefore scalar multiplication is from the right: $v\alpha$, where $v\in\mathbb{H}^n$, $\alpha\in\mathbb{H}$. Remember that multiplication in $\mathbb{H}$ is not commutative. Let $\mathbb{H}^{m\times n}$ denote the set of all $m\times n$ matrices in $\mathbb{H}$ and consider it as a left vector space. A matrix $A\in\mathbb{H}^{m\times n}$ can be interpreted as a linear transformation from $\mathbb{H}^{n}$ to $\mathbb{H}^m$. We will use the linear transformation and the matrix representing the linear transformation interchangeably.

Every quaternion $x\in\mathbb{H}$ has the form $x=x_0+x_1\I+x_2\J+x_3\K,$ where $x_i\in\mathbb{R}$ and the elements $\I,\J,\K$ satisfy the following formulas \begin{equation*}
\I^2=\J^2=\K^2=-1,\quad\I\J=-\J\I=\K,\quad\J\K=-\K\J=\I,\quad\K\I=-\I\K=\J.
\end{equation*}
Let $\bar{x}=x_0-x_1\I-x_2\J-x_3\K$ denote the conjugate quaternion of $x$. Let $A$ be an $m\times n$ quaternion matrix.  Then we denote  the $m\times n$ matrix in which each entry is the conjugate of the corresponding entry in $A$ by $\bar{A}$. The transpose of $\bar{A}$ is the $n\times m$ matrix  denoted by $A^*$. Note that every $n\times n$ quaternion matrix $A$ can be written as $A=A_1+\J A_2$, where $A_1$ and $A_2$ are $n\times n$ complex matrices.

\begin{definition}
A nonzero vector $v\in\mathbb{H}^n$ is called a \emph{(right) eigenvector} of a matrix $A\in\mathbb{H}^{n\times n}$ corresponding to the \emph{(right) eigenvalue} $\lambda\in\mathbb{H}$ if the equality $Av= v\lambda$ holds.
\end{definition}
An $n \times n$ quaternion matrix $A$ has left eigenvalues and right eigenvalues but we only  need right eigenvalues and right eigenvectors and therefore omit the word `right'. 

The \textit{spectrum} of $A$, denoted by $\sigma(A)$, is the set of all eigenvalues of $A$ and is closed under similarity of quaternions. From \cite{Zhang} we see that an $n\times n$ matrix $A$ has exactly $n$ eigenvalues which are complex numbers with nonnegative imaginary parts and the Jordan normal form of $A$ has precisely these numbers on the diagonal. Let $\mathbb{C}_+=\{\lambda\in\mathbb{C}\mid \text{Im}\,(\lambda)>0\}$ denote the open upper complex half-plane. The Jordan normal form of a quaternion matrix is a direct sum of Jordan blocks, $J_k(\lambda)$, corresponding to eigenvalue $\lambda\in\mathbb{C}_+$ and of size $k\times k$. We also need the standard involutary permutation (sip) matrix (a $k\times k$ matrix with ones on the main anti-diagonal and zeros elsewhere), and  denote it by $Q_k$. 

The null space (or kernel) and the range (or image) of the matrix $A\in\mathbb{H}^{m\times n}$ are as follows:
\begin{equation*}
{\rm Ker}\,A=\{x\in\mathbb{H}^n \mid Ax=0\};\quad \text{Im }A=\{Ax\mid x\in\mathbb{H}^n\}.
\end{equation*}
An $n\times n$ quaternion matrix $H$ is said to be \textit{Hermitian} if $H^*=H$ and \textit{skew-Hermitian} if $H^*=-H$. The eigenvalues of a Hermitian matrix are real, see Theorem~5.3.6(c) in \cite{Rodman}. Let $\pi(H)$ be the number of positive eigenvalues of the Hermitian matrix $H$. 

We consider the indefinite inner product $[\,\cdot\,,\cdot\,]$ defined by an invertible Hermitian matrix $H\in\mathbb{H}^{n\times n}$ as follows $[x,y]=\langle Hx,y\rangle=y^*Hx$ for $x,y\in\mathbb{H}^n$, where $\langle\cdot\,,\cdot\rangle$ denotes the standard inner product. It is important to note that the following is true for all $u,v\in\mathbb{H}^n$  and $\alpha\in\mathbb{H}$: 
\begin{equation*}
[x,y]^*=[y,x];\quad [x\alpha,y]=[x,y]\alpha;\quad [x,y\alpha]=\alpha^*[x,y].
\end{equation*}

A subspace $W$ of $\mathbb{H}^n$ is said to be \textit{$H$-nondegenerate} if $x\in W$ and $[x,y]=0$ for all $y\in W$ imply that $x=0$. Otherwise $W$ is $H$-degenerate. 

Let $W^{[\perp]}$ denote the \textit{$H$-orthogonal companion} of the subspace $W$ of $\mathbb{H}^n$, i.e.,
\begin{equation*}
W^{[\perp]}:=\{x\in\mathbb{H}^n\mid [x,y]=0\;{\rm for\; all}\;y\in W\}.
\end{equation*}

In terms of the $H$-orthogonal companion we have the following result for quaternion subspaces. See Proposition~3.6.4 in \cite{Rodman}.
\begin{proposition}\label{PropMnondeg}
Let $W$ be a subspace of $\mathbb{H}^n$. Then the following are equivalent:
\begin{enumerate}
\item[\rm (i)] $W$ is $H$-nondegenerate.
\item[\rm (ii)] $W^{[\perp]}$ is $H$-nondegenerate.
\item[\rm (iii)] $W^{[\perp]}$ is a direct complement to $W$ in $\mathbb{H}^n$.
\end{enumerate}
\end{proposition}

\begin{definition}
Let $H_1$ and $H_2$ be the matrices associated with the indefinite inner products $[\,\cdot\,,\cdot\,]_1$ on $\mathbb{H}^n$ and $[\,\cdot\,,\cdot\,]_2$ on $\mathbb{H}^m$, respectively. Let $X:\mathbb{H}^n\rightarrow\mathbb{H}^m$ be a linear transformation. Then $X^{[*]}:\mathbb{H}^m\rightarrow\mathbb{H}^n$ defined by 
\begin{equation*}
[X^{[*]}y,x]_1=[y,Xx]_2
\end{equation*} 
for all $x\in\mathbb{H}^n$, $y\in\mathbb{H}^m$, is called the \emph{$H_1$-$H_2$-adjoint} of $X$.
\end{definition}
Note that we can also write $X^{[*]}=H_1^{-1}X^*H_2$. In the case where $n=m$ and $H=H_1=H_2$, we call $X^{[*]}=H^{-1}X^*H$ the \textit{$H$-adjoint} of $X$. A matrix $A$ is said to be \textit{$H$-selfadjoint} if $A$ coincides with its $H$-adjoint. An $n\times n$ matrix $U$ is said to be \textit{$H$-unitary} if $[Ux,Uy]=[x,y]$ for all $x,y\in\mathbb{H}^n$, or equivalently, if $U^*HU=H$. In terms of the $H$-adjoint, we can also say $U$ is $H$-unitary if $U^{[*]}U=I$.

With $V$ and $W$ subspaces of $\mathbb{H}^n$, a linear transformation (or its matrix representation) $U:V\rightarrow W$ is called an \textit{$H$-isometry} if $[Ux,Uy]=[x,y]$ for all $x,y\in V$, or equivalently, if $U^*HUx=Hx$ for all $x\in V$. 

Every pair $(A,H)$ of quaternion matrices, where $A$ is $H$-selfadjoint, has a unique canonical form and it is interesting to note that it is identical to the canonical form of complex matrices. This is also given in, for example \cite[Theorem~4.1]{Alpay},  \cite{Karow} and \cite[Theorem~10.1.1]{Rodman}.
\begin{theorem}\label{ThmcanonformH}
Let $H\in\mathbb{H}^{n\times n}$ be an invertible Hermitian matrix and $A\in\mathbb{H}^{n\times n}$ an $H$-selfadjoint matrix. Then there exists an invertible matrix $S\in\mathbb{H}^{n\times n}$ such that
\begin{eqnarray}\label{eqcanonformH1}
 S^{-1}AS&=&J_{k_1}(\lambda_1)\oplus\cdots\oplus J_{k_\alpha}(\lambda_\alpha)\nonumber\\
 &\oplus & \begin{bmatrix}
J_{k_{\alpha+1}}(\lambda_{\alpha+1}) & 0 \\ 0 & J_{k_{\alpha+1}}(\overline{\lambda}_{\alpha+1}) 
 \end{bmatrix}\oplus\cdots\oplus
\begin{bmatrix}
J_{k_{\beta}}(\lambda_{\beta}) & 0 \\ 0 & J_{k_{\beta}}(\overline{\lambda}_{\beta}) 
 \end{bmatrix},
\end{eqnarray}
where $\lambda_i\in\sigma(A)\cap\mathbb{R}$ for all $i=1,\ldots,\alpha$, $\lambda_i\in\sigma(A)\cap\mathbb{C}_+$ for all $i=\alpha+1,\ldots,\beta$, and
\begin{equation}\label{eqcanonformH2}
 S^*HS=\eta_1Q_{k_1}\oplus\cdots\oplus \eta_\alpha Q_{k_\alpha}\oplus Q_{2k_{\alpha+1}}\oplus\cdots\oplus Q_{2k_\beta} ,
\end{equation}
where $\eta_i=\pm 1$. The form $(S^{-1}AS,\,S^*HS)$ in \eqref{eqcanonformH1} and \eqref{eqcanonformH2} is uniquely determined by the pair $(A,H)$, up to a permutation of diagonal blocks.
\end{theorem}


\section{$H$-selfadjoint square roots}\label{secSqRoots}

Necessary and sufficient conditions for the existence of a quaternion $H$-selfadjoint  $m$th root of a quaternion $H$-selfadjoint matrix were found in \cite{onsnr3}. The conditions are stated for matrices in the subalgebra 
\begin{equation*}
\Omega_{2n}:=\left\{\begin{bmatrix}
A_1 & \bar{A}_2 \\ -A_2 & \bar{A}_1
\end{bmatrix}\mid A_1,A_2\in\mathbb{C}^{n\times n}\right\}
\end{equation*} 
of $\mathbb{C}^{2n\times 2n}$, and  $\mathbb{H}^{n\times n}$ is isomorphic to $\Omega_{2n}$ by means of the isomorphism $\omega_n$ defined by
\begin{equation}\label{eqDefOmega_n}
\omega_n(A_1+\J A_2)=\begin{bmatrix}
A_1 & \bar{A}_2 \\ -A_2 & \bar{A}_1
\end{bmatrix}, \quad\text{where}\,\,A_1,A_2\in\mathbb{C}^{n\times n}.
\end{equation}
Then for any $n\times n$ quaternion matrix $B$ we have the following:
The canonical form of $(\omega_n(B),\omega_n(H))$ is given by $(J\oplus \bar{J},Q\oplus Q)$ if and only if the canonical form of $(B,H)$ is given by $(J,Q)$. To see this, apply $\omega_n$ to both sides of \eqref{eqcanonformH1} and \eqref{eqcanonformH2}, and use the properties $\omega_n(A^*)=(\omega_n(A))^*$ and $\omega_n(A^{-1})=(\omega_n(A))^{-1}$, where in the latter case, $A$ is invertible.

Now, for a given $H$-selfadjoint matrix $B$ in $\mathbb{H}^{n\times n}$ we present the necessary and sufficient conditions for the existence of an $H$-selfadjoint square root of $B$ (i.e., $m=2$ in Theorem~3.2 stated in \cite{onsnr3}).

\begin{theorem}\label{ThmSqRootExistence}
Let $B$ and $H$ be $n\times n$ quaternion matrices. Let $H$ be invertible and Hermitian, and let $B$ be $H$-selfadjoint. Then there exists an $H$-selfadjoint quaternion matrix, say $A$, such that $A^2=B$ if and only if the canonical form of $(B,H)$ has the following properties:
\begin{enumerate}
\item[\rm (i)] The part of the canonical form corresponding to the negative eigenvalues, say $(B_-,H_-)$, is given by
\begin{equation*}
B_-=\bigoplus_{j=1}^t\left(J_{k_j}(\lambda_j)\oplus J_{k_j}(\lambda_j)\right),\quad H_-=\bigoplus_{j=1}^t\left(Q_{k_j}\oplus -Q_{k_j}\right),
\end{equation*}
where $\lambda_j<0$.
\item[\rm (ii)] The part of the canonical form corresponding to the zero eigenvalues, say $(B_0,H_0)$, is given by 
\begin{equation*}
B_0=\bigoplus_{j=1}^tB^{(j)},\quad\quad H_0=\bigoplus_{j=1}^tH^{(j)}
\end{equation*}
where each corresponding pair of matrices $(B^{(j)},H^{(j)})$ is given by either 
\begin{equation*}
B^{(j)}=J_{a_j+1}(0)\oplus J_{a_j}(0),\quad H^{(j)}=\eta_j Q_{a_j+1}\oplus \eta_jQ_{a_j},
\end{equation*}
or
\begin{equation*}
B^{(j)}=J_{a_j}(0)\oplus J_{a_j}(0),\quad H^{(j)}=Q_{a_j}\oplus -Q_{a_j},
\end{equation*}
where $\eta_j=\pm 1$ and in the former case $a_{j}$ is allowed to be zero.
\end{enumerate}
\end{theorem}

\section{First steps towards $H$-polar decompositions}\label{secLem1stSteps}

Here we present a result for  quaternion matrices which is a modification of Lemma~4.1 stated in \cite{BR} for complex matrices. We will need this later in the proof of the conditions for an $H$-polar decomposition and it explains why we need a result on extensions of $H$-isometries to $H$-unitary matrices. To this end, a quaternion version of Witt's theorem is proved later.

\begin{lemma}\label{Lem4.1XYrelation}
Let $H_1\in\mathbb{H}^{n\times n}$ and $H_2\in\mathbb{H}^{m\times m}$ be the invertible Hermitian matrices which define indefinite inner products $[\,\cdot\,,\cdot\,]_1$ on $\mathbb{H}^n$  and $[\,\cdot\,,\cdot\,]_2$ on $\mathbb{H}^m$, respectively. Let $X$ and $Y$ be linear transformations from $\mathbb{H}^n$ to $\mathbb{H}^m$. Then $Y$ can be written in the form
\begin{equation*}\label{eqLemYX1}
Y=UX
\end{equation*}
where $U$ is an injective $H_2$-isometry from ${\rm Im}\,X$ to ${\rm Im}\,Y$, if and only if both
\begin{equation}\label{eqLemYX2}
Y^{[*]}Y=X^{[*]}X
\end{equation}
and
\begin{equation}\label{eqLemYX3}
{\rm Ker}\,X={\rm Ker}\,Y
\end{equation}
are satisfied.
\end{lemma}

Before we give the proof, here is an example as an illustration of why injectivity of $U$ is necessary in the lemma.
\begin{example}
Let $n=m=2$ and let $X$ and $Y$ be given as
\begin{equation*}
X=\begin{bmatrix}
1&0\\0&0
\end{bmatrix},\quad Y=\begin{bmatrix}
0&0\\0&0
\end{bmatrix} ,
\end{equation*} 
and let $H_1=H_2=\begin{bmatrix}
0 &1\\1&0
\end{bmatrix}$ be the matrix defining the indefinite inner product on $\mathbb{H}^2$. Then $\text{Im}\,X=\text{span}\left\{\begin{bmatrix}
1\\0
\end{bmatrix}\right\}$ and $\text{Im}\,Y=\left\{\begin{bmatrix}
0\\0
\end{bmatrix}\right\}$. Now suppose $U:\text{Im}\,X\rightarrow\text{Im}\,Y$ is a linear transformation such that $Y=UX$. For any $\alpha,\beta\in\mathbb{H}$, we have
\begin{equation*}
\left[U\begin{bmatrix}
\alpha \\0
\end{bmatrix},\,U\begin{bmatrix}
\beta\\0
\end{bmatrix}\right]_2=[0,0]_2=0
\end{equation*}
and
\begin{equation*}
\left[\begin{bmatrix}
\alpha \\0
\end{bmatrix},\,\begin{bmatrix}
\beta\\0
\end{bmatrix}\right]_2=\left\langle\begin{bmatrix}
0\\\alpha 
\end{bmatrix},\,\begin{bmatrix}
\beta\\0
\end{bmatrix}\right\rangle=0.
\end{equation*}
Therefore $U$ is an $H_2$-isometry, but note that $Ux=Uy=0$ for all $x,y\in\text{Im}\,X$ and it does not imply that $x=y$, so $U$ is not injective. Now,
\begin{equation*}
X^{[*]}X=H_1^{-1}X^*H_2X=\begin{bmatrix}
0 &1\\1&0
\end{bmatrix}\begin{bmatrix}
1&0\\0&0
\end{bmatrix}\begin{bmatrix}
0 &1\\1&0
\end{bmatrix}\begin{bmatrix}
1&0\\0&0
\end{bmatrix}=\begin{bmatrix}
0&0\\0&0
\end{bmatrix},
\end{equation*}
and $Y^{[*]}Y=\begin{bmatrix}
0&0\\0&0
\end{bmatrix}=X^{[*]}X$. However, we have 
\begin{equation*}
\text{Ker}\,Y=\mathbb{H}^2\quad\text{and}\quad\text{Ker}\,X=\text{span}\left\{\begin{bmatrix}
0\\1
\end{bmatrix}\right\},
\end{equation*}
thus the null spaces do not coincide.\qed
\end{example}

Now for the proof of Lemma~\ref{Lem4.1XYrelation}. The proof is essentially the same as in \cite{BR}, but a change to right scalar multiplication was needed.
\begin{proof}
Suppose that $Y=UX$ where $U$ is an injective $H_2$-isometry from $\text{Im}\,X$ to $\text{Im}\,Y$. Then 
\begin{equation*}
Y^{[*]}Y =X^{[*]}U^{[*]}UX=X^{[*]}X.
\end{equation*}
Since $U$ is injective, we have ${\rm Ker}\,Y={\rm Ker}\,UX={\rm Ker}\,X$. Therefore both \eqref{eqLemYX2} and \eqref{eqLemYX3} hold.

Conversely, assume that both \eqref{eqLemYX2} and \eqref{eqLemYX3} hold. We want to find an injective $H_2$-isometry between ${\rm Im}\,X$ and ${\rm Im}\,Y$. Let the dimension of ${\rm Im}\,X$ be $p$ and let $\{f_1,\ldots,f_p\}$ be a basis  of ${\rm Im}\,X\subseteq\mathbb{H}^m$. 
Now let $\{e_1,\ldots,e_p\}$ be vectors in $\mathbb{H}^n$ such that 
\begin{equation}\label{eqLemXe-f}
Xe_i=f_i,\quad i=1,2,\ldots,p,
\end{equation}
and let $g_i$ be the image of $e_i$ under $Y$, i.e.,
\begin{equation}\label{eqLemYe-g}
g_i=Ye_i, \quad i=1,2,\ldots,p.
\end{equation}
We prove that the set $\{g_1,\ldots,g_p\}$ is a basis of ${\rm Im}\,Y$. If 
\begin{equation*}
\sum_{i=1}^p g_i\gamma_i=0,
\end{equation*}
for some quaternion scalars $\gamma_i$, then it follows from \eqref{eqLemYe-g} that $\sum_{i=1}^p e_i\gamma_i\in{\rm Ker}\,Y$ and since \eqref{eqLemYX3} holds $\sum_{i=1}^p e_i\gamma_i\in{\rm Ker}\,X$. This implies that
\begin{equation*}
\sum_{i=1}^p f_i\gamma_i=X\sum_{i=1}^p e_i\gamma_i=0,
\end{equation*}
but since $\{f_1,\ldots,f_p\}$ is a basis, $\gamma_1=\gamma_2=\cdots=\gamma_p=0$. Therefore the vectors $\{g_1,\ldots,g_p\}$ are linearly independent. From \eqref{eqLemYX3} and the Rank Theorem (see Proposition~3.2.5(e) in \cite{Rodman}), we have that $\dim{\rm Im}\,Y=p$ and hence the $p$ vectors in \eqref{eqLemYe-g} form a  basis of ${\rm Im}\,Y$. Finally, define the linear map $U':{\rm Im}\,X\rightarrow{\rm Im}\,Y$ by 
\begin{equation*}
U'f_i=g_i,\quad i=1,2,\ldots,p.
\end{equation*}
Let $a=\sum_{i=1}^p f_i\alpha_i$ and $b=\sum_{i=1}^p f_i\beta_i$ be arbitrary vectors in ${\rm Im}\,X$, $\alpha_i,\beta_i\in\mathbb{H}$. Then by using \eqref{eqLemYX2}, \eqref{eqLemXe-f} and \eqref{eqLemYe-g} it can be shown that $[U'a,U'b]_2=[a,b]_2$:
\begin{eqnarray*}
[U'a,U'b]_2&=&[U'(f_1\alpha_1+\cdots+f_p\alpha_p),U'(f_1\beta_1+\cdots+f_p\beta_p)]_2\\
&=& [g_1\alpha_1+\cdots+g_p\alpha_p,g_1\beta_1+\cdots+g_p\beta_p]_2\\
&=& [Ye_1\alpha_1+\cdots+Ye_p\alpha_p,Ye_1\beta_1+\cdots+Ye_p\beta_p]_2\\
&=& [Y(e_1\alpha_1+\cdots+e_p\alpha_p),Y(e_1\beta_1+\cdots+e_p\beta_p)]_2\\
&=& [Y^{[*]}Y(e_1\alpha_1+\cdots+e_p\alpha_p),(e_1\beta_1+\cdots+e_p\beta_p)]_1\\
&=& [X^{[*]}X(e_1\alpha_1+\cdots+e_p\alpha_p),(e_1\beta_1+\cdots+e_p\beta_p)]_1\\
&=& [Xe_1\alpha_1+\cdots+Xe_p\alpha_p,Xe_1\beta_1+\cdots+Xe_p\beta_p]_2\\
&=& [a,b]_2
\end{eqnarray*}
 and thus  $U'$ is an $H_2$-isometry. Assume $U'a=0$, then $\sum_{i=1}^p g_i\alpha_i=0$ but the vectors $\{g_1,\ldots,g_p\}$ are linearly independent and therefore $a=0$. This implies that $U'$ is injective and we conclude the proof.
\end{proof}


\section{Witt's theorem}\label{secWittsThm}
Witt's theorem is known for vector spaces over the real and complex numbers and gives conditions for the extension of isometries between subspaces of $\mathbb{F}^n$ (where $\mathbb{F}$ is the field of real or complex numbers) to the whole of $\mathbb{F}^n$. Witt's theorem is an indispensable tool in the study of $H$-polar decompositions. We extend Witt's theorem to the quaternion case by supplying a proof that is essentially the same as the proof of Theorem~2.1 in \cite{BMRRR3}.

\begin{theorem}\label{ThmWittExistence}
Let $H_1,H_2\in\mathbb{H}^{n\times n}$  be invertible Hermitian matrices which define two inner products $[\,\cdot\,,\cdot\,]_1$ and $[\,\cdot\,,\cdot\,]_2$ on $\mathbb{H}^n$, respectively.
Assume that $\pi(H_1)=\pi(H_2)$. Let $U_0:V_1\rightarrow V_2$ be a nonsingular linear transformation, where $V_1$ and $V_2$ are subspaces in $\mathbb{H}^n$,  such that
\begin{equation} \label{eqWittU_0iso}
[U_0x,U_0y]_2=[x,y]_1\quad\textit{for every}\quad x,y\in V_1.
\end{equation}
Then there exists a linear transformation $U:\mathbb{H}^n\rightarrow\mathbb{H}^n$ such that
\begin{equation}\label{eqWittUiso}
[Ux,Uy]_2=[x,y]_1\quad\textit{for every}\quad x,y\in\mathbb{H}^n
\end{equation}
and
\begin{equation}\label{eqWittU_0=U}
Ux=U_0x\quad\textit{for every}\quad x\in V_1.
\end{equation}
\end{theorem}

Any linear transformation $U$ with the property \eqref{eqWittUiso} is called \textit{$H_1$-$H_2$-unitary} and any $H_1$-$H_2$-unitary linear transformation that also satisfies \eqref{eqWittU_0=U} is called a \textit{Witt extension} of $U_0$.

\begin{proof}
Let $m$ be the dimension of $V_1$ and let $\{e_1,\ldots,e_m\}$ be a basis of $V_1\subseteq\mathbb{H}^n$ such that 
\begin{equation*}
[e_i,e_j]_1=\begin{cases}
\;\;\,1 & {\rm if}\;\; i=j=m_0+1,m_0+2,\ldots,m_0+m_+, \\
-1 &{\rm if}\;\; i=j=m_0+m_++1,m_0+m_++2,\ldots,m, \\
\;\;\,0 & {\rm otherwise},
\end{cases}
\end{equation*}
where $m_0+m_++m_-=m$. 
Thus the Hermitian matrix describing the restriction of the inner product $[\,\cdot\,,\cdot\,]_1$ to $V_1$ has $m_+$ positive eigenvalues and $m_-$ negative eigenvalues and the multiplicity of zero is $m_0$. 

The strategy that we follow is to construct a basis for $\mathbb{H}^n$ by using the basis for $V_1$ and we start by forming $H_1$-nondegenerate subspaces where each contains one of the first $m_0$ basis vectors. Define a functional $\alpha_i:\mathbb{H}^n\rightarrow\mathbb{H}$ for every $i=1,2,\ldots,m$ as follows:
\begin{equation*}
\alpha_i(x)=[x,e_i]_1, \quad i=1,2,\ldots,m.
\end{equation*}
Since $\alpha_1,\ldots,\alpha_m$ are linearly independent, there exists vectors $\tilde{e}_i\in\mathbb{H}^n$ such that $\alpha_i(\tilde{e}_j)=\delta_{ij}$, where $\delta_{ij}=1$ if $i=j$ and $\delta_{ij}=0$ if $i\neq j$, i.e.\ $[\tilde{e}_j,e_i]_1=\delta_{ij}$ for all $i=1,2,\ldots,m$. Then let 
\begin{equation*}
W_k={\rm span}\,\{e_k,\tilde{e}_k\},\quad k=1,2,\ldots,m_0
\end{equation*}
and since $[e_k,e_k]_1=0$ and $[\tilde{e}_k,e_k]_1=1$, each $W_k$ is $H_1$-nondegenerate. Note that for $\beta_k=-\textstyle\frac{1}{2}[\tilde{e}_k,\tilde{e}_k]_1$, we have
\begin{eqnarray*}
[\tilde{e}_k+e_k\beta_k,\tilde{e}_k+e_k\beta_k]_1
&=&[\tilde{e}_k,\tilde{e}_k]_1+[e_k\beta_k,e_k\beta_k]_1+[\tilde{e}_k,e_k\beta_k]_1+[e_k\beta_k,\tilde{e}_k]_1\\
&=& [\tilde{e}_k,\tilde{e}_k]_1+\beta_k^*[e_k,e_k]_1\beta_k+\beta_k^*[\tilde{e}_k,e_k]_1+[e_k,\tilde{e}_k]_1\beta_k\\
&=&[\tilde{e}_k,\tilde{e}_k]_1-\frac{1}{2}[\tilde{e}_k,\tilde{e}_k]_1-\frac{1}{2}[\tilde{e}_k,\tilde{e}_k]_1=0,
\end{eqnarray*}
and 
\begin{equation*}
[e_k,\tilde{e}_k+e_k\beta_k]_1=[e_k,\tilde{e}_k]_1+\beta_k^*[e_k,e_k]_1=[e_k,\tilde{e}_k]_1.
\end{equation*}
Therefore we can always replace the vector $\tilde{e}_k$ by $\tilde{e}_k+e_k(-\frac{1}{2}[\tilde{e}_k,\tilde{e}_k]_1)$ and thus without loss of generality we can assume that $[\tilde{e}_k,\tilde{e}_k]_1=0$ for $k=1,2,\ldots,m_0$. Now, let 
\begin{equation*}
e_k'=(e_k-\tilde{e}_k)\frac{1}{\sqrt{2}}, \quad\quad e_k''=(e_k+\tilde{e}_k)\frac{1}{\sqrt{2}}.
\end{equation*}
Simple calculations analogous to those above give the following
\begin{equation*}
[e_k',e_k']_1=-1,\quad [e_k'',e_k'']_1=1\quad\text{and}\quad [e_k',e_k'']_1=0.
\end{equation*}

It follows that the subspace $W=W_1+\cdots +W_{m_0}+\text{span }\{e_j\}_{i=m_0+1,\ldots,m}$ is $H_1$-nondegenerate and thus by Proposition~\ref{PropMnondeg} the $H_1$-orthogonal companion $W^{[\perp]}$ of $W$ is  $H_1$-nondegenerate and $W^{[\perp]}$ is a direct complement to $W$ in $\mathbb{H}^n$. Therefore, we can append the vectors $e_s$ for $s=2m_0+m_++m_-+1,2m_0+m_++m_-+2,\ldots,n$ to the set
\begin{equation*}
\{e'_1,\ldots,e'_{m_0},e_1'',\ldots,e''_{m_0},e_{m_0+1},e_{m_0+2},\ldots,e_m\}
\end{equation*}
of $2m_0+m_++m_-$ vectors such that the resulting ordered set $\{g_1,\ldots,g_n\}$ will be a basis in $\mathbb{H}^n$ with the property 
\begin{equation*}
[g_i,g_j]_1=\varepsilon_i\delta_{ij},\;\text{ for } i,j=1,2,\ldots,n,\text{ where }\varepsilon_i=\pm 1.
\end{equation*}

As a last step before we can define the extension, we look at the subspace $V_2$. Let $f_i=U_0e_i$ for $i=1,2,\ldots,m$. We introduce vectors $f_k'$ and $f_k''$  (for $k=1,2,\ldots,m$) and vectors $f_s$ (for $s=2m_0+m_++m_-+1,2m_0+m_++m_-+2,\ldots,n$) in the same way we introduced the vectors $e_k'$, $e_k''$ and $e_s$ but using $[\,\cdot\,,\cdot\,]_2$ instead of $[\,\cdot\,,\cdot\,]_1$. Then this will result in a basis $\{h_1,\ldots,h_n\}$ of $\mathbb{H}^n$. The hypotheses  $\pi(H_1)=\pi(H_2)$ and \eqref{eqWittU_0iso} in the theorem statement ensure that $[h_i,h_j]_2=[g_i,g_j]_1$ for all $i,j=1,\ldots,n$.

Finally, we define the $n\times n$ matrix $U$ by the equalities
\begin{equation*}
\begin{array}{cll}
Ue_k'=f_k',   &&\text{for}\;\;k=1,\ldots,m_0, \\
Ue_k''=f_k'', &&\text{for}\;\;k=1,\ldots,m_0, \\
Ue_j=f_j ,     &&\text{for}\;\;j=m_0+1,\ldots,m, \\
Ue_s=f_s,      &&\text{for}\;\;s=2m_0+m_++m_-+1,\ldots,n.
\end{array}
\end{equation*}
From the construction above it is easy to see that the matrix $U$ satisfies both \eqref{eqWittUiso} and \eqref{eqWittU_0=U} and is therefore a Witt extension of $U_0$.
\end{proof}

We next include an extended Witt's theorem  which gives a description of any Witt extension of a given $U_0$. The proof in the quaternion case is once again essentially the same as the proof of Theorem~2.3 in \cite{BMRRR3}. 

Firstly, we mention the following: By using the same notation (and vectors) as in the proof of Theorem~\ref{ThmWittExistence}, let
\begin{eqnarray}\label{eqEbasis}
&\mathcal{E}=\{e_1,\ldots,e_{m_0},e_{m_0+1},\ldots,e_m,\tilde{e}_1,\ldots,\tilde{e}_{m_0},e_{2m_0+m_++m_-+1},\ldots,e_n\},\\
&\mathcal{F}=\{Ue_1,\ldots,Ue_m,U\tilde{e}_1,\ldots,U\tilde{e}_{m_0},Ue_{2m_0+m_++m_-+1},\ldots,Ue_n\}.\label{eqFbasis}
\end{eqnarray}
Note that the subspaces $V_1$ and $V_2$ are spanned by the first $m$ vectors of the bases of $\mathcal{E}$ and $\mathcal{F}$, respectively. The Gramian matrix of the basis $\mathcal{E}$ with respect to $[\,\cdot\,,\cdot\,]_1$  (and of the basis $\mathcal{F}$ with respect to $[\,\cdot\,,\cdot\,]_2$) is equal to
\begin{equation}\label{eqH_1GramianE+F}
\begin{bmatrix}
 0    &  0   & I_{m_0} & 0  \\
 0    & J_1  &   0   &   0  \\
I_{m_0}&  0  &   0   &   0  \\
 0    &   0  &   0   &  J_2
\end{bmatrix},
\end{equation}
where $J_1$ and $J_2$ are both diagonal matrices with $+1$ and $-1$ on its diagonal. The first $m_+$ diagonal entries of $J_1$ are $+1$ and the remaining $m_-$ diagonal entries are $-1$. The matrix $J_2$ is the Gramian matrix of the last $n-m-m_0$ vectors in \eqref{eqEbasis} and we assume without loss of generality that it has said form. 

\begin{theorem}\label{ThmWittDescription}
Let $\tilde{U}$ be a Witt extension of the $n\times n$ matrix $U_0$ as in Theorem~\ref{ThmWittExistence}. Then there exists a $J_2$-unitary matrix $P_1$ (of order $n-m-m_0$), an $(n-m-m_0)\times m_0$ matrix $P_2$, and a skew-Hermitian $m_0\times m_0$ matrix $P_3$, such that the matrix of $\tilde{U}$ has the form 
\begin{equation}\label{eqThmWittextForm}
\tilde{U}=\begin{bmatrix}
I_{m_0}&  0  & -\frac{1}{2}P_2^*J_2P_2+P_3 & -P_2^*J_2P_1 \\
 0    & I_{m-m_0} &     0           &    0   \\
 0    &   0      &   I_{m_0}      &    0    \\
 0    &   0     &     P_2       &    P_1
\end{bmatrix}
\end{equation}
Here $m=\dim V_1$ and $m_0$ is the number of zero eigenvalues of the Gramian matrix of any basis in $V_1$ with respect to $[\,\cdot\,,\cdot\,]_1$.

Conversely, if $P_1$ is an arbitrary $J_2$-unitary matrix, $P_2$ is an arbitrary $(n-m-m_0)\times m_0$ matrix, and $P_3$ is an arbitrary skew-Hermitian $m_0\times m_0$ matrix, then the matrix $\tilde{U}$ defined by \eqref{eqThmWittextForm} is a Witt extension of $U_0$.
\end{theorem}

\begin{proof}

To ensure that $\tilde{U}x=U_0x$ for all $x\in V_1$, any extension $\tilde{U}$ of $U_0$ in the bases \eqref{eqEbasis} and \eqref{eqFbasis} has the form
\begin{equation}\label{eqWittGenForm}
\tilde{U}=\begin{bmatrix}
I & 0 & A_1 & A_2 \\
0 & I & A_3 & A_4 \\
0 & 0 & A_5 & A_6 \\
0 & 0 & A_7 & A_8
\end{bmatrix},
\end{equation}
for some matrices $A_i$ with sizes the same as the corresponding blocks in \eqref{eqH_1GramianE+F}. 
The key to proving the theorem lies in the following: the matrix \eqref{eqWittGenForm} is $H_1$-$H_2$-unitary if and only if $H_1^{-1}\tilde{U}^*H_2\tilde{U}=I$.  By using \eqref{eqH_1GramianE+F} and \eqref{eqWittGenForm}, a simple computation shows that $H_1^{-1}\tilde{U}^*H_2\tilde{U}=I$ holds if and only if
\begin{equation}\label{eqWittLHS=RHS}
\begin{bmatrix}
 A_5^*  &   A_3^*J_1  & u_{13} & u_{14} \\
   0    &     I     &   A_3  &   A_4  \\
   0    &     0     &   A_5  &   A_6  \\
J_2A_6^*&J_2A_4^*J_1& u_{43} & u_{44}
\end{bmatrix}=\begin{bmatrix}
I & 0 & 0 & 0 \\
0 & I & 0 & 0 \\
0 & 0 & I & 0 \\
0 & 0 & 0 & I
\end{bmatrix},
\end{equation}
where
\begin{eqnarray*}
u_{13}&=&A_5^*A_1+A_3^*J_1A_3 + A_1^*A_5 + A_7^*J_2A_7,    \\
u_{14}&=&A_5^*A_2 + A_3^*J_1A_4 + A_1^*A_6 + A_7^*J_2A_8,  \\
u_{43}&=&J_2A_6^*A_1 + J_2A_4^*J_1A_3 + J_2A_2^*A_5 + J_2A_8^*J_2A_7, \\
u_{44}&=&J_2A_6^*A_2 + J_2A_4^*J_1A_4 + J_2A_2^*A_6 + J_2A_8^*J_2A_8.
\end{eqnarray*}
By equating the corresponding blocks in \eqref{eqWittLHS=RHS} we obtain  the following
\begin{equation}\label{eqWitt2Solutions}
A_5=I,\quad A_3=A_4=A_6=0,\quad A_2=-A_7^*J_2A_8, \quad A_8^*J_2A_8=J_2\quad \text{and}\quad A_1+A_1^*=-A_7^*J_2A_7.
\end{equation} 
When we write $A_1=\frac{1}{2}(A_1+A_1^*)+\frac{1}{2}(A_1-A_1^*)$, where the first term is Hermitian and the second is skew-Hermitian, we can use the last equality  in \eqref{eqWitt2Solutions} to find an expression for $A_1$. Then by taking $P_1=A_8$, $P_2=A_7$ and $P_3=\frac{1}{2}(A_1-A_1^*)$, the matrix \eqref{eqThmWittextForm} in the theorem statement is derived and the proof is complete.
\end{proof}


\section{$H$-polar decompositions}\label{secMainPolarDecomp}
Finally, we are ready to give necessary and sufficient conditions for the existence of an $H$-polar decomposition of a given quaternion matrix.

\begin{theorem}\label{ThmFINALPolardecomp}
Let $H$ be an invertible Hermitian matrix in $\mathbb{H}^{n\times n}$ and let $X$ be a given ${n\times n}$ quaternion matrix. Then $X$ admits an $H$-polar decomposition, say $X=UA$ for an $H$-selfadjoint $A$ and an $H$-unitary $U$, if and only if there exists an $H$-selfadjoint square root $A$ of $X^{[*]}X$  with $\textup{Ker}\,X=\textup{Ker}\,A$.
\end{theorem}

\begin{proof}
Assume that there exists an $n \times n$ $H$-unitary matrix $U$ and an $n\times n$ $H$-selfadjoint matrix $A$ such that $X=UA$. Then 
\begin{equation*}
X^{[*]}X=(UA)^{[*]}UA=A^{[*]}U^{[*]}UA=AA=A^2,
\end{equation*}
i.e., $A$ is an $H$-selfadjoint square root of $X^{[*]}X$. Also note that $U$ is invertible and therefore 
\begin{equation*}
\text{Ker}\,X=\text{Ker}\,UA=\text{Ker}\,A.
\end{equation*}

Conversely, suppose that there exists an $H$-selfadjoint square root $A$ of $X^{[*]}X$, i.e.\ $X^{[*]}X=A^2$, such that $\textup{Ker}\,X=\textup{Ker}\,A$. Since $A$ is  $H$-selfadjoint, we can write  $X^{[*]}X=A^{[*]}A$. Then by Lemma~\ref{Lem4.1XYrelation} there exists an injective $H$-isometry $U_0$ from $\text{Im}\,A$ to $\text{Im}\,X$ such that $X=U_0A$. Note that the conditions in Theorem~\ref{ThmWittExistence} are satisfied where $V_1=\text{Im}\,A$ and $V_2=\text{Im}\,X$ and therefore we can form a Witt extension of the $H$-isometry to the whole space $\mathbb{H}^n$. That is, there exists a matrix $U\in\mathbb{H}^{n\times n}$ such that $[Ux,Uy]=[x,y]$ for all $x,y\in\mathbb{H}^n$ and $Ux=U_0x$ for all $x\in\text{Im}\,A$. Hence, $U$ is $H$-unitary and $X=U_0A=UA$ which means that $X$ admits an $H$-polar decomposition.
\end{proof}

By using Theorem~\ref{ThmSqRootExistence} we can rewrite the criterion for the existence of an $H$-polar decomposition as follows (similarly to Theorem~4.4 in \cite{BMRRR}).
%

\begin{theorem}\label{ThmFINALPolardecomp2}
Let $H\in\mathbb{H}^{n\times n}$ be an invertible Hermitian matrix. Then a given $X\in\mathbb{H}^{n\times n}$ admits an $H$-polar decomposition if and only if all the following conditions are satisfied.
\begin{enumerate}
\item[\rm (i)] Each block in the canonical form of $(X^{[*]}X,H)$ that corresponds to a negative eigenvalue $\lambda_i$ of $X^{[*]}X$ is of the form
\begin{equation*}
\left(J_{k_i}(\lambda_i)\oplus J_{k_i}(\lambda_i),\,Q_{k_i}\oplus -Q_{k_i}\right).
\end{equation*}
\item[\rm (ii)] Each block in the canonical form of $(X^{[*]}X,H)$ that corresponds to the zero eigenvalue of $X^{[*]}X$ is either of the form
\begin{equation*}
(B_i,H_i)=\left(J_{k_i+1}(0)\oplus J_{k_i}(0),\,\eta_iQ_{k_i+1}\oplus \eta_iQ_{k_i}\right),
\end{equation*}
or of the form
\begin{equation*}
(B_i,H_i)=\left(J_{k_i}(0)\oplus J_{k_i}(0),\,Q_{k_i}\oplus -Q_{k_i}\right),
\end{equation*}
where $\eta_i=\pm 1$ and where $k_i$ is allowed to be zero in the former case. We use  $B_0$ for the $k_0\times k_0$ zero matrix, i.e. for all the single blocks $J_1(0)$, and $H_0$ for the corresponding $k_0\times k_0$ diagonal matrix with $+1$ and $-1$ on the diagonal.
\item[\rm (iii)] Let $\ell_i$ denote the order of a block $B_i$. There is a choice of basis ${\{e_{i,j}\}_{i=0}^m}_{j=1}^{\ell_i}$ in $\mathbb{H}^n$ which produces the canonical form in the second assertion and 
\begin{eqnarray}\label{eqKerXconditionnr3}
\textup{Ker}\,X&=&\textup{span}\{e_{i,1}+e_{i,k_1+1}\mid\ell_i=2k_i,i=1,\ldots,m\}\\
&\oplus&\textup{span}\{e_{i,1}\mid\ell_i=2k_i-1,i=1,\ldots,m\}\oplus\textup{span}\{e_{0,j}\}_{j=1}^{k_0}.\nonumber
\end{eqnarray}
\end{enumerate}
\end{theorem}

\begin{proof}
Let $X\in\mathbb{H}^{n\times n}$ be given and suppose that assertions (i) to (iii) in the theorem statement hold. We want to prove that there exists an $H$-selfadjoint square root $A$ of $X^{[*]}X$ for which $\textup{Ker}\,X=\textup{Ker}\,A$ holds. Then by Theorem~\ref{ThmFINALPolardecomp}, the matrix $X$ admits an $H$-polar decomposition and the proof is complete. Since (i) and (ii) are satisfied, Theorem~\ref{ThmSqRootExistence} implies that $X^{[*]}X$ has an $H$-selfadjoint square root. 
The strategy that we now follow is to construct for each block $B_i$  as in (ii) of the theorem statement, an $H_i$-selfadjoint matrix $A_i$ such that $A_i^2=B_i$ and $\textup{Ker}\,A_i=\textup{Ker}\,X\cap\,\text{span}\{e_{i,j}\}_{j=1}^{\ell_i}$. Firstly, let $B_i$ be of even size, say $\ell_i=2k_i$, $k_i\geq 1$. As in the proof of Theorem~4.4 in \cite{BMRRR}, let $S_i$ be the matrix with columns 
\begin{equation*}
\textstyle (e_{i,1}+e_{i,k_i+1})\frac{1}{\sqrt{2}},\,(e_{i,1}-e_{i,k_i+1})\frac{1}{\sqrt{2}},\,(e_{i,2}+e_{i,k_i+2})\frac{1}{\sqrt{2}},\,(e_{i,2}-e_{i,k_i+2})\frac{1}{\sqrt{2}},\ldots,
(e_{i,k_i}+e_{i,2k_i})\frac{1}{\sqrt{2}},\,(e_{i,k_i}-e_{i,2k_i})\frac{1}{\sqrt{2}}.
\end{equation*}
Then $A_i=S_iJ_{2k_i}(0)S_i^{-1}$ is an $H_i$-selfadjoint square root of $B_i$ and 
\begin{equation*}
\text{Ker}\,A_i=\text{span}\{e_{i,1}+e_{i,k_i+1}\}=\text{Ker}\,X\,\cap\,\text{span}\{e_{i,j}\}_{j=1}^{\ell_i}.
\end{equation*}
Secondly, let $B_i$ be of odd size, say $\ell_i=2k_i-1$, $k_i\geq 1$. Again, as in \cite{BMRRR}, let $S_i$ be the matrix with columns
\begin{equation*}
e_{i,1},\,e_{i,k_i+1},\,e_{i,2},\,e_{i,k_i+2},\ldots,\,e_{i,k_i-1},\,e_{i,2k_i-1},\,e_{i,k_i}.
\end{equation*}
Then $A_i=S_iJ_{2k_i-1}(0)S_i^{-1}$ is an $H_i$-selfadjoint square root of $B_i$ and 
\begin{equation*}
\text{Ker}\,A_i=\text{span}\{e_{i,1}\}=\text{Ker}\,X\,\cap\,\text{span}\{e_{i,j}\}_{j=1}^{\ell_i}.
\end{equation*}
Hence if $A$ is the $H$-selfadjoint square root of $B$ consisting of a direct sum of all the $A_i$'s, then we have that $\text{Ker}\,X=\text{Ker}\,A$.

Conversely, suppose that $X\in\mathbb{H}^{n\times n}$ admits an $H$-polar decomposition, say $X=UA$, where $U\in\mathbb{H}^{n\times n}$ is $H$-unitary and $A\in\mathbb{H}^{n\times n}$ is $H$-selfadjoint. By Theorem~\ref{ThmFINALPolardecomp} the $H$-selfadjoint $A$ is a square root of the $H$-selfadjoint matrix $X^{[*]}X$ for which  $\text{Ker}\,X=\text{Ker}\,A$ holds. Since $X^{[*]}X$ has an $H$-selfadjoint square root, the canonical form of the pair $(X^{[*]}X,H)$ satisfies the conditions in Theorem~\ref{ThmSqRootExistence}. Thus assertions (i) and (ii) hold and there exists a choice of basis ${\{e_{i,j}\}_{i=0}^m}_{j=1}^{\ell_i}$ in $\mathbb{H}^n$ which provides the canonical form as given in (ii). By using this basis and a construction for $A$ as was done above, one can easily see that $\text{Ker}\,A$ is equal to the right hand side of \eqref{eqKerXconditionnr3} and since $\text{Ker}\,X=\text{Ker}\,A$, we conclude the proof.
\end{proof} 


\begin{remark}
Since there exists an isomorphism between $\mathbb{H}^{n\times n}$ and $\Omega_{2n}$, all of the results and specifically the conditions for the existence of an $H$-polar decomposition, are also true in $\Omega_{2n}$. In a previous paper, \cite{onsnr3}, proofs were given for matrices in $\Omega_{2n}$ as well as an explanation that they are also true for matrices in $\mathbb{H}^{n\times n}$ via the isomorphism $\omega_n$ as defined in \eqref{eqDefOmega_n}. Here, however, we took a direct approach and proved the results for quaternion matrices.  
\end{remark}

\bigskip
{\bf Acknowledgments.} This work is based on research supported in part by the DSI-NRF Centre of Excellence in Mathematical and Statistical Sciences (CoE-MaSS). Opinions expressed and conclusions arrived at are those of the authors and are not necessarily to be attributed to the CoE-MaSS.



\begin{thebibliography}{1}
\bibitem{Alpay}
D.~Alpay, A.C.M.~Ran and L.~Rodman. 
\newblock  Basic classes of matrices with respect to quaternionic indefinite inner product spaces.
\newblock  {\it Linear Algebra Appl.,} 416:242--269, 2006.

\bibitem{BMRRR} 
{Y.~Bolshakov, C.V.M.~van~der~Mee, A.C.M.~Ran, B.~Reichstein and L.~Rodman}. 
\newblock {Polar decompositions in finite dimensional indefinite scalar product spaces: General theory}.
\newblock {\it Linear Algebra Appl.,} {216}:91--141, 1997.

\bibitem{BMRRR2}
{Y.~Bolshakov, C.V.M.~van der Mee, A.C.M.~Ran, B.~Reichstein and L.~Rodman}.
\newblock {Polar decompositions in finite dimensional indefinite scalar product spaces: Special cases and applications}.
\newblock {\it Recent Developments in Operator Theory and its Applications},
{87}:61--94, 1996.

\bibitem{BMRRR2err}
{Y.~Bolshakov, C.V.M.~van der Mee, A.C.M.~Ran, B.~Reichstein and L.~Rodman}.
\newblock {Errata for: Polar decompositions in finite dimensional indefinite scalar product spaces: Special cases and applications}.
\newblock {\it Integr.\ Equ.\ Oper.\ Theory},
{27}:497--501, 1997.

\bibitem{BMRRR3}
{Y.~Bolshakov, C.V.M.~van~der~Mee, A.C.M.~Ran, B.~Reichstein and L.~Rodman}.
\newblock {Extension of isometries in finite-dimensional indefinite scalar product spaces and polar decompositions}.
\newblock {\it Linear Algebra Appl.,}
{18(3)}:752--774, 1997.

\bibitem{BR}
{Y.~Bolshakov and B.~Reichstein}.
\newblock {Unitary equivalence in an indefinite scalar product: an analogue of singular-value decomposition}.
\newblock {\it Linear Algebra Appl.,}
{222}:155--226, 1995.

\bibitem{onsnr1}
{G.J.~Groenewald, D.B.~Janse~van~Rensburg, A.C.M.~Ran, F.~Theron and M.~van~Straaten.}
\newblock  {$m$th roots of $H$-selfadjoint matrices}.
\newblock  \emph{Linear Algebra Appl.,} 610:804--826, 2021.

\bibitem{onsnr3}
{D.B.~Janse~van~Rensburg, A.C.M.~Ran, F.~Theron and M.~van~Straaten}.
\newblock  {$m$th roots of $H$-selfadjoint matrices over the quaternions}.
\newblock  \emph{Electron.\ J.\ Linear Algebra,}
Preprint, To appear: 2021.

\bibitem{Karow}
{M.~Karow.}
\newblock  Self-adjoint operators and pairs of Hermitian forms over the quaternions.
\newblock  {\it Linear Algebra Appl.,}  299:101--117, 1999.

\bibitem{MRR} 
{C.V.M.~van der Mee, A.C.M.~Ran and L.~Rodman}.
\newblock {Stability of self-adjoint square roots and polar decompositions in indefinite scalar product spaces}
\newblock {\it Linear Algebra Appl.}, {302-303}:77--104, 1999.

\bibitem{Rodman}
{L.~Rodman.}
\newblock  {\it Topics in Quaternion Linear Algebra}.
\newblock  {Princeton University Press,} Princeton, 2014.

\bibitem{Zhang}
F.~Zhang.
\newblock  Quaternions and Matrices of Quaternions.
\newblock  {\it Linear Algebra Appl.,} 251:21--57, 1997.

\bibitem{ZhangWei}
F.~Zhang and Y.~Wei.
\newblock  Jordan canonical form of a partitioned complex matrix and its application to real quaternion matrices.
\newblock  {\it Comm. Algebra,} 29:2363--2375, 2001.

\bibitem{Wiegmann}
N.A.~Wiegmann.
\newblock Some theorems on matrices with real quaternion elements.
\newblock {\it Canad.\ J.\ Math.,} 7:191--201, 1955.


\end{thebibliography}
\end{document}